\documentclass[11pt]{article}
\usepackage{amsmath, amssymb, amsthm, verbatim,enumerate,bbm}
\usepackage{indentfirst}

\title{Packing and counting arbitrary Hamilton cycles in random digraphs}

\author{Asaf Ferber\thanks{Department of Mathematics, Yale University, and Department of Mathematics, MIT, USA.\newline \hspace*{1.5em} Email: asaf.ferber@yale.edu,
and ferbera@mit.edu.}\and Eoin Long \thanks{School of Mathematical Sciences, Tel Aviv University, Tel Aviv, Israel. Email:
eoinlong@post.tau.ac.il.}}

\date{\today}
\allowdisplaybreaks

\theoremstyle{plain}
\newtheorem{theorem}{Theorem}[section]
\newtheorem{lemma}[theorem]{Lemma}

\newtheorem{corollary}[theorem]{Corollary}
\newtheorem{conjecture}[theorem]{Conjecture}

\newcommand{\Bin}{\ensuremath{\textrm{Bin}}}

\begin{document}
\maketitle

\begin{abstract}
    We prove packing and counting theorems for arbitrarily oriented
    Hamilton cycles in ${\cal D}(n,p)$ for nearly optimal $p$ (up to a $\log ^cn$ factor). In particular, we show that
    given $t = (1-o(1))np$ Hamilton cycles $C_1,\ldots ,C_{t}$, each of which is oriented arbitrarily, a digraph
    $D \sim {\cal D}(n,p)$ w.h.p. contains edge disjoint copies of $C_1,\ldots ,C_t$, provided $p =\omega(\log ^3 n/n)$. We also show that given an
    arbitrarily oriented $n$-vertex cycle $C$, a random digraph
    $D \sim {\cal D}(n,p)$ w.h.p. contains $(1\pm o(1))n!p^n$ copies of $C$,
    provided $p \geq \log ^{1 + o(1)}n/n$.
\end{abstract}

\section{Introduction}

A \emph{Hamilton cycle} in a graph or digraph is a cycle
passing through every vertex of the graph exactly once, and a graph
is \emph{Hamiltonian} if it contains a Hamilton cycle. Hamiltonicity
is one of the most central notions in graph theory, and has been
intensively studied by numerous researchers in recent decades.

One of the first, and probably most celebrated, sufficient conditions
for Hamiltonicity in graphs was established by Dirac \cite{Dirac} in 1952. He proved that every graph on $n$ vertices, $n\ge 3$, with minimum
degree at least $n/2$ is Hamiltonian. Ghouila-Houri \cite{Ghouila} proved an analogue of Dirac's theorem for digraphs by showing that any digraph of minimum \emph{semi-degree} at least $n/2$ contains an oriented Hamilton cycle (the semi-degree of a digraph $G$, denoted $\delta^0(G)$, is the minimum of all the in- and out-degrees of vertices of $G$).

Instead of studying ``consistently oriented" Hamilton cycles in digraphs, it is natural to consider Hamilton cycles with arbitrary orientations. This problem goes back to the $80$s where Thomasson \cite{Thom} showed that each regular tournament contains every orientation of a Hamilton cycle. Later on, H\"aggkvist and Thomason \cite{Hagg} showed an approximate analog of the result of
Ghouila-Houri \cite{Ghouila} while proving that $\delta^0(G)\geq n/2+n^{5/6}$ is sufficient to guarantee every orientation of a Hamilton cycle to appear in $G$. Very recently, this problem has been settled completely by DeBiasio, K\"uhn, Molla, Osthus and Taylor \cite{Debiasio}. They showed that $\delta^0(G)\geq n/2$ is enough for all cases other than an \emph{anti-directed} Hamilton cycle, where for the latter, Debiaso and Molla showed in \cite{Debiasio1} that $\delta^0(G)\geq n/2+1$ is enough (an anti-directed Hamilton cycle is a cycle with no two consecutive edges having the same orientation).

In this paper we restrict our attention to the sparse setting, that is, to \emph{random directed graphs}. Let $\mathcal D(n,p)$ be the probability space consisting of all directed graphs on vertex set
$[n]$ in which each possible arc is added with probability
$p$ independently at random.

 One of the first results regarding Hamilton cycles in random directed graphs was obtained by McDiarmid in \cite{McDiarmid}.  He showed (among other things) by using an elegant coupling argument that
$$\Pr[G\sim \mathcal G(n,p) \textrm{ is
    Hamiltonian}]\leq \Pr[D\sim \mathcal D(n,p) \textrm{ is
    Hamiltonian}].$$ Combined with the result of Bollob\'as
\cite{bollobas1984evolution} it follows that a typical $D\sim
\mathcal D(n,p)$ is Hamiltonian for $p\geq \frac{\ln n+\ln\ln
    n+\omega(1)}n$.
Frieze \cite{frieze1988algorithm} later determined the correct threshold for the appearance of a Hamilton cycle in $D\sim \mathcal D(n,p)$ is $p=\frac{\ln n+\omega(1)}{n}$.

While Frieze's result gives a better bound than McDiarmid's coupling argument, the latter is much more flexible (for some further applications, see \cite{Closing gaps}).
For example, given an arbitrary oriented Hamilton cycle $C$, it follows immediately from McDiarmid's proof that
$$\Pr[G\sim \mathcal G(n,p) \textrm{ is
    Hamiltonian}]\leq \Pr[D\sim \mathcal D(n,p) \textrm{ contains
    a copy of }C].$$
    In contrast, the result of Frieze is tailored to ``consistently oriented" Hamilton cycles and gives no improvement for the obtained bound of $p=\frac{\ln n+\ln\ln n+\omega(1)}{n}$. It may be interesting to find the exact threshold for the appearance of an arbitrary oriented Hamilton cycle and we conjecture the following:
\begin{conjecture}
  Let $C$ be a Hamilton cycle oriented arbitrarily, then a digraph $D\sim \mathcal D(n,p)$ w.h.p.\ contains a copy of $C$, provided that $p=\frac{\ln n+\omega(1)}{n}$.
\end{conjecture}

Another recent result worth mentioning was given by Ferber, Nenadov, Peter, Noever and {\v{S}}koric in \cite{ferberrobust}. Here it was proven using the ``absorption method" that $D\sim \mathcal D(n,p)$ is w.h.p. Hamiltonian even if an adversary deletes roughly one half of the in- and out-degrees of all the vertices, provided that $p\geq polylog n/n$.

 Here we deal with the problems of counting and packing arbitrary oriented Hamilton cycles in $D\sim \mathcal D(n,p)$, for edge-densities $p\geq polylog(n)/n$. The analogous problems regarding the ``consistently oriented" Hamilton cycles has been recently treated by Kronenberg and the authors in \cite{FerKroLong}. However, the proof method there is inapplicable to the arbitrary oriented case.

Enhancing a recent ``online sprinkling" technique introduced by Ferber and Vu \cite{FV}, we manage to tackle these two problems. Our first theorem gives an asymptotically optimal result for packing arbitrarily oriented Hamilton cycles in $D \sim {\cal D}(n,p)$.

\begin{theorem}
    \label{thm: packing theorem}
    Let $\epsilon >0$ and $p(n)\in (0,1]$. Let
    $t = (1 - \epsilon )np$ and suppose that $C_1,\ldots , C_t$ are $n$-vertex
    cycles with arbitrary orientations. Then w.h.p.
    $D \sim {\cal D}(n,p)$ contains edge disjoint copies of $C_1,\ldots ,C_t$,
    provided $p \gg {\log ^3n}/{n} $.
\end{theorem}

Our second result shows that given an arbitrarily oriented Hamilton cycle $C$, w.h.p.
$D\sim \mathcal D(n,p)$ contains the ``correct" number of copies of $C$.

\begin{theorem}
  \label{thm: Counting}
 Suppose that $C$ is an arbitrarily oriented $n$-vertex cycle. Then w.h.p. a digraph
 $D\sim \mathcal D(n,p)$ contains $(1\pm o(1))^nn!p^n$ distinct copies
 of $C$, provided $p \gg (\log \log n)\log n/n$.
\end{theorem}

Before closing the introduction, let us mention that packing and counting Hamilton cycles in the undirected setting has been extensively studied by
numerous researchers. In fact, both of these problems are now completely resolved (see \cite{glebov2013number,knox2013edge, krivelevich2012optimal, kuhn2014hamilton} and their references). The main difficulty when working in the directed case, is that the so called Pos\'a rotation-extension technique (see \cite{posa}) does not work in its simplest form and therefore one should find more creative ways for generating Hamilton cycles.
\vspace{1.5mm}

\noindent \textbf{Notation:} Given a directed graph (\emph{digraph}) $D$, we write $V(D)$ for the vertex set of $D$ and $E(D)$ for the edge set $D$. Given $v \in V(D)$ we write $N^+(v) = \{u\in V(D): \overrightarrow {vu} \in E(D)\}$, the \emph{out-neighbourhood} of $v$, and let $d^+(v) = |N^+(v)|$, the \emph{outdegree} of $v$ in $D$. Similarly define $N^-(v)$ and $d^-(v)$. Let $\delta ^0(D)$ denote the \emph{semi-degree} of $D$, given by $\delta ^0(D) = \min _{v\in V(D), *\in \{+,-\}} d^*(v)$. Given $n\in {\mathbb N}$, let $D_n$ denote the \emph{complete directed graph} (or  \emph{complete digraph}) on $n$ vertices, consisting of all possible $n(n-1)$ directed edges.

A \emph{path} $P$ of length $k$ is a $(k+1)$-vertex digraph with $k$ edges, given by $P:=v_0v_1\ldots v_k$ where for each $i\in [0,k-1]$ either $\overrightarrow{v_i v_{i+1}}$ or $\overleftarrow{v_i v_{i+1}}$ is an edge of $P$. Given $\sigma:[0,k]
\rightarrow \{+,-\}$, we say that $P$ is a $\sigma$-\emph{path}, if for all $i\in [0,k-1]$ the edge $\overrightarrow{v_iv_{i+1}}$ lies in $P$ whenever $\sigma(i)=+$, and $\overleftarrow{v_iv_{i+1}}$ lies in $P$ whenever $\sigma(i)=-$. In this case we write $\sigma (P) = \sigma $. In a similar way, for $\sigma:[1,k]\rightarrow\{+,-\}$ a
$\sigma$-\emph{cycle} $C:=v_1\ldots v_kv_1$ is a $k$-vertex digraph with $k$ edges, each of the form $\overrightarrow{v_iv_{i+1}}$ or $\overleftarrow{v_iv_{i+1}}$, where each appears according to the sign of $\sigma(i)$. Given a cycle $C$ and a subpath $P$ of $C$, let $P^c$ denote the path induced by the edges of $C$ which do not lie in $P$, called the \emph{compliment} of $P$ in $C$.

Given a digraph $D$, we write ${\cal D}(D,p)$ for the probability space of random subdigraphs of $D$ obtained by including each edge of $D$ independently with probability $p$. For a graph $G$, we write $\mathcal G(G,p)$ for the analogous distribution on subgraphs of $G$. In the special case when $D = D_n$ we simply write $\mathcal D(n,p)$. Similarly for graphs we write $\mathcal G(n,p)$. When discussing events related to ${\cal D}(D,p)$ or ${\cal G}(G,p)$ we write \emph{with high probability} (w.h.p.) to denote that the event occurs with probability $1-o(1)$. We say \emph{with very high probability} (w.v.h.p.) to mean with probability $1-n^{-\omega(1)}$.

\section{Tools}

\subsection{A concentration inequality}

A sequence of random variables $0 \equiv X_0, X_1,\ldots ,X_N$ is said a submartingale if $${\mathbb E}(X_i|X_1,\ldots ,X_{i-1}) \leq 0 \text{ for all }i\in [N].$$ The next result gives a concentration bound for submartingales (see Theorem 27 in the survey of Chung and Lu \cite{Chung-Lu}, taking $\phi _i = a_i = 0$).

\begin{theorem}
    \label{thm: submartingale concentration}
    Suppose that $0 \equiv X_0,X_1,\ldots ,X_N$ is a submartingale
    which for all $i\in [N]$ satisfies
        \begin{equation*}
            {\mathbb V}ar(X_i|X_1,\ldots X_{i-1}) \leq \sigma;
            \qquad \mbox{ and } \qquad
            X_i - {\mathbb E}(X_i|X_1,\ldots X_{i-1}) \leq M.
        \end{equation*}
    Then ${\mathbb P}(X_N \geq m) \leq e^{-{m^2}/{2(N\sigma + M m/3)}}$.
\end{theorem}

We will make use of the following simple corollary.

\begin{corollary}
    \label{corollary: control of submartingale}
    Suppose that $A_1,\ldots ,A_N$ are a sequence of events in a
    probability space $(\Omega ,{\mathbb P})$. Suppose that for
    all $i\in [N]$ we have ${\mathbb P}(A_i|A_1,\ldots ,A_{i-1})
    \leq q$. Then letting $E_m$ denote the event that
    at least $qN + m$ of the events $A_1,\ldots ,A_N$ occur, we
    have ${\mathbb P}(E_m) \leq e^{-{m^2}/{2(Nq + m/3)}}$.
\end{corollary}

\begin{proof}
    Simply take $Y_i$ to be the indicator random variable for the
    event $A_i$, and set $X_i = Y_i -q$ for all $i\in [N]$, and $X_0 \equiv 0$.
    Then $${\mathbb E}(X_i|X_1,\ldots ,X_{i-1}) =
    {\mathbb P}(A_i| A_1,\ldots ,A_{i-1}) - q \leq 0 \text{ for all }i\in [N],$$ and so $X_0,\ldots ,X_N$ forms a submartingale.
    We also have
$${\mathbb V}ar(X_i|X_1,\ldots X_{i-1}) \leq q
    \text{ for all }i.$$ Taking $M = 1$ and $\sigma = q$,
    Theorem \ref{thm: submartingale concentration}
    gives
        \begin{equation*}
            {\mathbb P}\big (E_m)
                =
            {\mathbb P}\big (\sum _{i=1}^{N} Y_i > nq + m \big )
                =
            {\mathbb P} \big (\sum _{i=1}^{N} X_i > m \big )
                \leq
            e^{-{m^2}/{2(Nq + m/3)}},
        \end{equation*}
    as required.
\end{proof}

\subsection{Completing paths into a Hamilton cycle}

We will also use the following result to complete paths into Hamilton cycles
in $D\sim \mathcal D(n,p)$.

\begin{lemma}
  \label{finishing a long path}
Let $\sigma \in \{+,-\}^n$. Suppose that $G$ is an $n$-vertex digraph with
$\delta ^0(G) = (1-o(1))n$. Then a digraph $D\sim {\mathcal D}(G,p)$ contains a $\sigma $-path $Q$ between \emph{any} two distinct vertices in $D$, provided $p=\omega(\log n/n)$.
\end{lemma}

The analogue of Lemma \ref{finishing a long path} is well known for undirected graphs states that w.h.p. $G \sim {\cal G}(n,p)$ contains a Hamiltonian path between any two distinct vertices $a, b \in V(G)$, provided $p = \omega (\log n/n)$. (In fact, this result follows from many of the proofs of Hamiltonicity of $G \sim {\cal G}(n,p)$.) This undirected result can be `upgraded' to the directed world to prove Lemma \ref{finishing a long path} using the nice coupling argument of McDiarmid \cite{McDiarmid} mentioned in the introduction.

\section{Packing arbitrarily orientated Hamilton cycles in ${\cal D}(n,p)$}

In this section we prove Theorem \ref{thm: packing theorem}. The proof naturally splits into two pieces. In the first piece, which appears in the next subsection, we will describe and analyse a simple randomized embedding algorithm to generate long paths of some fixed orientation. Then, in subsection 3.2 by repeatedly running this embedding algorithm in ${\cal D}(n,p)$ we will find a large subpath from each cycle $C_i$. Combined with an additional argument to close each of these paths to a cycle, this will prove the Theorem \ref{thm: packing theorem}.

\subsection{A randomized algorithm for embedding oriented paths}

Let $D$ be an $n$-vertex digraph with $\delta ^0(D) \geq n - \Delta$.
Also let $P= v_1\cdots v_{\ell }$ be a $\sigma $-path, for some arbitrary $\sigma : [\ell - 1] \to \{+,-\}$. Our aim in this section is to describe a randomized algorithm which w.h.p. finds a copy $Q:=x_1\cdots x_{\ell}$ of $P$ in $D$ over $\ell$ rounds. Throughout the algorithm, whenever we `expose an edge' we mean to toss a biased coin with certain probability for this edge.
\vspace{1mm}

\textbf{Path embedding algorithm:}
\begin{enumerate}
    \item To begin, select a vertex $x_1 \in V(D)$ uniformly at random and set $Q_1 = x_1$.
    \item For $1\leq i\leq \ell-1$: suppose we are in round $i$ and that
    we  have now found $Q_{i} = x_1\cdots x_{i}$, and aim
    to extend it to $Q_{i+1}$ by finding $x_{i+1}$.
    Let $R_i = V(D) \setminus V(Q_{i})$ and select an ordering
    of $R_i$ uniformly at random, say $y_1,\ldots ,y_{n-i}$.
    \item To find $x_{i+1}$ proceed as follows. First expose $x_{i}y_1$ with an orientation corresponding to $\sigma(i)$,
    with probability
    $p_{ex}$. If this pair is exposed as an edge \underline{and}
    is an edge of $D$, set $x_{i+1} = y_1$ and $Q_{i+1} = Q_{i}x_{i+1}$.
    Otherwise expose $x_{i}y_2$ with an orientation corresponding to $\sigma(i)$,
    with probability $p_{ex}$. Again, if the exposed pair appears
    as an edge \underline{and} is an edge of $D$, set $x_{i+1} = y_2$ and
    $Q_{i+1} = Q_{i}x_{i+1}$. Continue with this process until we either find
    $x_{i+1}$ and $Q_{i+1}$, or run out of vertices in $R_i$. If this
    second case occurs, terminate the algorithm and declare a failure.
    If there is no failure and $i< \ell -1$ return to 2. for
    round $i+1$. Otherwise, proceed to 4.

    \item Output $Q:=Q_{\ell}$.
\end{enumerate}

To analyze the algorithm, we will be interested in the following events:
    \begin{align*}
        &F = \mbox{``the algorithm fails"};\\
        &E_{\overrightarrow {uv}}  = \mbox{``the edge }\overrightarrow{uv}
        \mbox{ is exposed during the algorithm"};\\
        &A_{{u,v}} = ``\{u,v\} \cap
        \big (V(Q) \setminus \{x_1,x_{\ell}\} \big ) = \emptyset".
    \end{align*}

The following lemma collects a number of key properties of this embedding process.

\begin{lemma}
        \label{lem: control of probabilities during path embedding}
    Let $D$ be a digraph with $\delta ^0(D) \geq n - \Delta $ and
    let $P$ be a path of length $\ell $.
    Suppose $p_{ex}$ and $\ell $ satisfy ${\log n}/(n-\ell - \Delta ) \ll
    p_{ex} \ll \min \big \{  \frac{(n-\ell )^2}{n^2 \Delta } ,
    \frac{1}{(n\Delta )^{1/2}} \big \}$.
    Then running the path embedding algorithm with $p_{ex}$
    to find a copy of $P$ in $D$, we have:
        \begin{enumerate}[(i)]
            \item $\Pr[F]=o\left({n^{-2}}\right)$.
            \item $\Pr\left[E_{\overrightarrow {uv}}\right] \leq
            \frac{1 + o(1)}{np_{ex}}$
            for every pair $\{u,v\} \in \binom{V(D)}{2}$.
            \item $\Pr\left[A_{u,v}\right] \leq (1 + o(1))
            \big ( \frac {n-\ell }{n} \big )^2$
            for every pair $\{u,v\} \in \binom{V(D)}{2}$.
        \end{enumerate}
\end{lemma}

\begin{proof}
    We first prove \emph{(i)}. Note that the algorithm only ends in failure if
    for some $i \in [\ell -1]$ edges of orientation $\sigma(i)$ in $E(D)$
    between $x_i$ and all vertices of $R_{i}$ were exposed, but none
    appeared as an edge. Using $|R_i| \geq n - \ell$, we see that
    $$\Pr[F] \leq
    \ell (1-p_{ex})^{n-\ell - \Delta} \leq ne^{-(n-\ell - \Delta)p_{ex}} = o(n^{-2}),$$
where the last inequality holds since $p_{ex} =\omega\left( \frac {\log n}{n-\ell -\Delta }\right)$.

    To see \emph{(ii)} and \emph{(iii)} it is helpful to think of the algorithm
    as proceeding in a slightly different, but equivalent way. First select
    a random subdigraph $G$ of $D_n$, where each directed edge of $D_n$ appears independently in
    $G$ with probability $p_{ex}$. Now simply run the original algorithm
    to find a copy of $P$, but this time instead of exposing edges with probability $p_{ex}$, we add the edge if the corresponding edge is present in $G$.
    Clearly this gives an identical distribution on paths which appear as $Q$.

    Now we claim that for all vertices $u_1,u_2,v_1,v_2 \in V(D)$ with $u_i \neq v_i$
    for $i = 1,2$
        \begin{equation}
            \label{equation: examined probabilities}
            \Pr\left[E_{\overrightarrow{u_1v_1}}\right] \leq
            \Pr\left[E_{\overrightarrow{u_2v_2}}\right] + (8\Delta + 12)p_{ex},
        \end{equation}
    and that
        \begin{equation}
            \label{equation: probability differences}
            \Pr\left[A_{u_1,v_1}\right] \leq \Pr\left[A_{u_2,v_2}\right] + (8\Delta + 12)p_{ex}.
        \end{equation}
    To see this, first note that if both $u_1$ and $u_2$ have the same in and out-neighbourhoods in $D$, and $v_1$ and $v_2$ have the same in and out-neighbourhoods in $D$, then
    $\Pr \left [ E_{\overrightarrow {u_1v_1}} \right ] =
    \Pr \left [ E_{\overrightarrow {u_2v_2}} \right ]$ and
    $\Pr\left[A_{u_1,v_1}\right] = \Pr\left[A_{u_2,v_2}\right]$. The key observation to
    proving \eqref{equation: examined probabilities} and \eqref{equation: probability differences} is that conditional on a high probability event, we can assume that this `same neighbourhood' property holds.

    Concretely, let
    $$S =
    \big \{z \in V(D): \text{ at least one of the edges }\overrightarrow{yz}, \overrightarrow{zy} \text{ is not in }D \mbox { for some } y\in
    \{u_1,u_2,v_1,v_2\} \big \} .$$
That is, $S$ is the set of vertices which are not in-neighbours or out-neighbours in $D$ of at least one vertex from $\{u_1,u_2,v_1,v_2\}$. Now consider the following event
        \begin{align*}
            B = ``\mbox{no edge }\overrightarrow{yz} \mbox{ or } \overrightarrow{zy}
            \mbox{ appears in } G, \mbox { where }
             y\in \{u_1,u_2,v_1,v_2\} \mbox{ and } z \in S \cup \{u_1,u_2,v_1,v_2\}".
        \end{align*}
Note that conditional on $B$, by symmetry of the neighbourhoods of $u_1, u_2, v_1$ and $v_2$, the path embedding algorithm is equally likely to expose the
${\overrightarrow {u_1v_1}}$ and the edge ${\overrightarrow {u_2v_2}}$, i.e. $\Pr \left [ E_{\overrightarrow {u_1v_1}} \right | B] =
\Pr \left [ E_{\overrightarrow {u_2v_2}} \right | B ]$. But this gives
        \begin{align}
            \label{equation:exposing edges is roughly the same}
                \Pr\left[E_{\overrightarrow{u_1v_1}}\right]
              &=
                \Pr\left[E_{\overrightarrow{u_1v_1}}|B\right]
                \Pr[B]
                    +
                \Pr\left[E_{\overrightarrow{u_1v_1}}|B^c\right]
                \Pr[B^c]\nonumber \\
                 & \leq
                \Pr\left[E_{\overrightarrow{u_2v_2}}|B\right]
                \Pr[B]
                    +
                \Pr[B^c]\nonumber  \\
                & \leq
                \Pr\left[E_{\overrightarrow {u_2v_2}}\right] + \Pr[B^c].
        \end{align}
    Similarly, conditional on $B$, by symmetry, the pairs $\{u_1,v_1\}$ and $\{u_2,v_2\}$ are
    equally likely to be disjoint from $V(Q) \setminus \{x_1,x_{\ell }\}$ and $\Pr[A_{u_1,v_1}|B] = \Pr[A_{u_2,v_2}|B]$.
Therefore, an identical
    calculation to \eqref{equation:exposing edges is roughly the same} gives
        \begin{align*}
                \Pr[A_{{u_1,v_1}}]
             \leq
                \Pr[A_{{u_2v_2}}] + \Pr[B^c].
        \end{align*}
    But $\Pr[B^c] \leq (8\Delta + 12)p_{ex}$, as each $u \in V(D)$ has
    $\leq \Delta $ non in-neighbours, $\leq \Delta $ non out-neighbours
    in $D$ and there are at most $12$ edges between vertices in
    $\{u_1,u_2,v_1,v_2\}$ in $D$. This gives
    \eqref{equation: examined probabilities} and
    \eqref{equation: probability differences}.

    Now we can prove \emph{(ii)}. For each $\overrightarrow{uv} \in E(D)$, let
    $C_{\overrightarrow{uv}}$ denote
    the event $$C_{\overrightarrow{uv}}=``\overrightarrow{uv} \text{ gets exposed during the algorithm and }\overrightarrow{uv}\in E(G)".$$ Clearly, we have $\Pr[C_{\overrightarrow{uv}}]
    = p_{ex} \times \Pr[E_{\overrightarrow{uv}}]$. Let $X$ denote the
    random variable which counts the number of edges in $G \cap D$ which get successfully exposed.
    Using \eqref{equation: examined probabilities}, for any edge
    $\overrightarrow{uv} \in E(D)$ we have
        \begin{align}
                \label{equation:control of exposed edges probabilities}
            {\mathbb E}(X)
                &=
            \sum _{{\overrightarrow{xy}} \in E(D)}
            \Pr[C_{\overrightarrow{xy}}] \nonumber \\
                &\geq
            \sum _{{\overrightarrow{xy}} \in E(D)}
            p_{ex} \times \big ( \Pr[E_{\overrightarrow{uv}}]\nonumber - (8\Delta + 12)p_{ex} \big )\\
                &
            \geq \big (n(n-1) - \Delta n \big ) p_{ex} \times
            \big (\Pr[E_{\overrightarrow{uv}}] - (8\Delta  + 12)p_{ex} \big )
            \nonumber \\
                &
            =   (1-o(1))n^2p_{ex} \times \big ( \Pr[E_{\overrightarrow{uv}}]
            - (8\Delta  + 12)p_{ex} \big ).
        \end{align}
    However we always have $X \leq \ell $, as each successfully exposed edge in $G \cap D$ completes a
    round of the algorithm and the algorithm consists of at most $\ell$ rounds. Combined with
    \eqref{equation:control of exposed edges probabilities} this gives
    $$\Pr[E_{\overrightarrow{uv}}] \leq
    (1 + o(1))\frac {\ell }{n^2p_{ex}} + (8 \Delta + 12)p_{ex} = (1 + o(1))\frac {1}{np_{ex}},$$ since $\ell \leq n$ and
    $p_{ex} =o\left((n\Delta )^{-1/2}\right)$. By applying
    \eqref{equation: examined probabilities} again, we conclude that
    $\Pr[E_{\overrightarrow{xy}}] \leq (1 + o(1))\frac {1}{np_{ex}}$ for all distinct
    $x,y \in V(D)$, completing  \emph{(ii)}.

    Lastly, it is left to prove \emph{(iii)}. Let $Y$ denote the random variable which counts the
    number of pairs $\{u,v\}$ with $\{u,v\} \cap \big ( V(Q) \setminus \{x_1,x_{\ell }\}
    \big ) = \emptyset $. Note that we always have $Y \leq \binom {n}{2}$
    and that if $F^c$ holds then $Y = \binom {n-\ell +2}{2}$. Since by \emph{(i)} we have
    $\Pr[F]=o(1/n)$, it therefore follows that
        \begin{equation*}
            {\mathbb E}(Y)
                \leq
             \Pr[F]\binom {n}{2} +
             \Pr[F^c]\binom {n - \ell + 2}{2}
                \leq
             (1 + o(1))\frac {(n-\ell)^2}{2}.
        \end{equation*}
    But for distinct $u,v \in V(D)$, from  \eqref{equation: probability differences}
    we have
        \begin{align*}
            {\mathbb E}(Y)
                & =
            \sum _{\{x,y\} \in \binom {[n]}{2}} \Pr[A_{x,y}] \geq
            \binom {n}{2} \big (\Pr[A_{u,v}]
            - (8\Delta + 12)p_{ex} \big ).
        \end{align*}
    Rearranging, we obtain $\Pr[A_{u,v}] \leq (1 + o(1))
    \Big ( \frac {n-\ell }{n} \Big )^2 + (8\Delta +12 )p_{ex} = (1 + o(1))
    \Big ( \frac {n-\ell }{n} \Big )^2$, since $p_{ex} \ll
    \frac {1}{\Delta } \big ( \frac {n-\ell }{n} \big )^2$, as required.
\end{proof}

\subsection{Finding edge disjoint Hamilton cycles in ${\cal D}(n,p)$}

    In this subsection we prove Theorem \ref{thm: packing theorem}.
    In the proof, it will be useful to assume that $p$ is not too large
    at certain points in the argument. We will assume that $\log ^3 n/n
    \ll p \leq n^{-2/3}$. The general situation can be reduced to this
    as follows. If $p \geq n^{-2/3}$ let $p' = n^{-5/6}$ and $k = n^{5/6}p $,
    so that $n^{1/6} \ll k \leq n^{5/6}$. Then, after generating
    $D \sim {\cal D}(n,p)$, we further partition $D$ into one of $k$ subdigraphs
    $D_1,\ldots ,D_k$, where each edge $e$ is assigned to $D_i$ with
    probability $1/k$.  It is clear that each $D_i$ is distributed as
    ${\cal D}(n, p')$. Therefore, if we prove that the statement of Theorem
    holds w.v.h.p (probability $1 - o(1/n)$) when $\log ^3 n/n
    \ll p \leq n^{-2/3}$, by taking a union bound over all the digraphs $D_1, \ldots ,D_k$ above, we prove that w.h.p. for all $p \gg \log ^3 n /n$.
    \vspace{1mm}

    Suppose that $p = \alpha ^3\log ^3 n/n$. Since $\log ^3/n \ll p \leq
    n^{-2/3}$ we have $1\ll \alpha \leq n^{1/9}$. Also let
    $\ell = n - n/ \alpha \log n$, $\Delta = n^{1/3}$ and $p _{ex}
    = \alpha ^2 \log ^2 n/n$. Let $C_1,\ldots ,C_t$ be cycles as given
    in the statement. Also set $p_1 = (1-\epsilon /2)p$ and choose $p_2$
    so that $(1-p_1)(1-p_2)=1-p$. Note that $p_2 =
    (1+o_{\epsilon }(1))\epsilon p/2$. Also, from each $C_i$ we select an
    oriented subpath $P_i$ of length $\ell $, with orientation $\sigma _i$.

    Our general plan is to embed the paths $\{P_i\}_{i\in [t]}$ into $\mathcal D(n,p_1)$ by repeatedly applying the algorithm described in the previous section. We will then expose new edges with probability $p_2$, to complete each copy of $P_i$ into a copy of the cycle $C_i$. Of course, we ensure that the obtained cycles are edge-disjoint. The embedding scheme proceeds in two stages.
    \vspace{1mm}

    \noindent
    \textbf{Stage 1:} Finding edge disjoint copies of $P_1,\ldots ,P_t$ in
    $D_1 \sim {\cal D}(n,p_1)$
    \vspace{1mm}

    In this stage we give a randomized algorithm which w.h.p. finds
    edge disjoint copies of $P_1,\ldots ,P_t$ in $D_1 \sim
    {\cal D}(n,p_1)$. The algorithm exposes edges from $D_n$ in
    an online fashion with some small probability $p_{ex}=o(p_1)$. This will be carried out over a series of rounds, where in
    round $i$ we aim to create a copy of $P_i$.

    Initialize to Round 1 and proceed as follows:
        \begin{enumerate}
            \item In round $i$ we have copies of $P_1,\ldots ,P_{i-1}$,
            denoted $Q_1,\ldots ,Q_{i-1}$. Set $D^{(i)} =
            D_n \setminus {(\cup _{j<i} E(Q_j))}$.

            \item Apply the path embedding algorithm from the previous subsection
            to find a copy of $P_i$ in $D^{(i)}$, denoted by
            $Q_i = x_{i,1}\ldots x_{i,\ell }$.  If this subroutine fails, declare
            a failure and terminate the algorithm.

            \item If $i < t$ then return to 1. in round $i+1$ to find $Q_{i+1}$.
            If $i = t$ then we have succeeded in finding edge disjoint copies of
            $P_1,\ldots ,P_t$.
        \end{enumerate}
    Given distinct vertices $u,v \in V(D_n)$, consider the following random variables:
        \begin{align*}
            X_{\overrightarrow{uv}}
                 =
            \big |\big \{i\in [t]: \overrightarrow{uv} \mbox{ is exposed
            during round }i \big \}\big |; \quad
            Y_{u,v}
                 =
            \big |\big \{ i\in [t]: \{u,v\} \cap
            \big (V(Q_i) \setminus \{x_{i,1}, x_{i,\ell }\}\big ) = \emptyset \big \}
            \big |.
        \end{align*}
    We claim that the following three properties hold with probability
    $1 - o(n^{-1})$:
        \begin{enumerate}
            \item [(\emph{a})] The algorithm succeeds in finding copies of $P_1,\ldots ,P_t$;
            \item [(\emph{b})] $X_{\overrightarrow{uv}} \leq \frac {p_1}{p_{ex}}$
            for all distinct $u,v \in V(D)$;
            \item [(\emph{c})] $Y_{u,v} \leq (1 + \epsilon ) \times t \times \big ( \frac {n-\ell }{n} \big )^{2}$ for all distinct $u,v \in V(D)$.
        \end{enumerate}
    This will complete Stage 1 of the algorithm. Indeed, by (\emph{b}) w.h.p. we have
    exposed each edge with probability $p_{ex} \times X_{\overrightarrow {uv}}
    \leq p_1$. The resulting digraph can therefore be coupled as a subgraph of
    $D_1 \sim {\cal D}(n,p_1)$
    w.h.p.. By (\emph{a}) this subdigraph contains edge disjoint copies of
    $P_1,\ldots ,P_t$, as required. (Property (\emph{c}) will be required
    for Stage 2.)

    To see (\emph{a}) note that in round $i$, each vertex has at most $2$
    neighbours in each $Q_j$ for $j<i$, and therefore $\delta ^0(D^{(i)})
    \geq n- 2(i-1) \geq n - 2t \geq n - 2\Delta $. Note that
        \begin{align*}
            \frac {\log n}{n-\ell -\Delta } = \frac {\log n}{
            n/ \alpha \log n - n^{1/3}}
                \leq
            \frac {2\alpha \log ^2n}{n}
                \ll
            p_{ex}
                \ll
            p &\leq
            \min \big \{ \frac {1}{\alpha ^2 \log ^2 n (n^{1/3})} ,
            n^{-2/3} \big \}\\
             & \leq
            \min \big \{  \frac{(n-\ell )^2}{n^2 \Delta } ,
    \frac{1}{(n\Delta )^{1/2}} \big \}.
        \end{align*}
    Therefore, by Lemma \ref{lem: control of probabilities during path embedding}
    \emph{(i)} the path embedding algorithm succeeds in round $i$ with probability
    at least $1 - o(\frac {1}{n^2})$. Therefore, by a union bound, the algorithm
    succeeds in producing a copy of $P_1,\ldots ,P_t$ with probability
    $1- o(n^{-1})$.

    We now prove (\emph{b}). Given distinct
    vertices $u,v \in V(D_n)$ we have $X_{\overrightarrow{uv}} = \sum _{i\in [t]}
    X_{\overrightarrow{uv}}(i)$, where $X_{\overrightarrow{uv}}(i)$ denotes the
    indicator random variable of the event that $\overrightarrow{uv} $ is
    exposed in during round $i$. Note that from Lemma
    \ref{lem: control of probabilities during path embedding}
    \emph{(ii)}, conditional on any of choice of $D^{(i)}$, we have
    $$\Pr\left[X_{\overrightarrow{uv}}(i) = 1|D^{(i)} \right] \leq
    (1+\epsilon /4)\frac{1}{np_{ex}}.$$ By Lemma
    \ref{corollary: control of submartingale}, we have
    $$\Pr \big [X_{\overrightarrow{uv}} \geq (1+ \epsilon /2)\frac{t}{np_{ex}} \big ] \leq e^{-\epsilon ^2 t/(64np_{ex})}=
    o(1/n^3).$$
    This holds as $t/np_{ex} \geq np/np_{ex} \gg \log n$. Therefore, with probability $1 - o(n^{-1})$ we have $X_{\overrightarrow{uv}}
    \leq (1+ \epsilon /2) \frac {t}{np_{ex}} = (1+ \epsilon /2)
    \frac {(1-\epsilon )np}{np_{ex}} \leq  \frac {p_1}{p_{ex}}$ for all
    $\overrightarrow{uv} \in E(D_n)$. This proves (\emph{b}).

    Lastly, (\emph{c}) is similar to (\emph{b}). Given distinct
    $u,v \in V(D_n)$ we have $Y_{{u,v}} = \sum _{i\in [t]}
    Y_{{u,v}}(i)$, where $Y_{{u,v}}(i)$ denotes the
    indicator random variable of the event $\{u,v\} \cap \big ( V(Q_i) \setminus
    \{x_{i,1}, x_{i,\ell }\} \big ) = \emptyset $. By Lemma
    \ref{lem: control of probabilities during path embedding} \emph{(iii)},
    conditional on any choice of $D^{(i)}$, we have
    $$\Pr[Y_{{u,v}}(i) = 1|D^{(i)}] \leq (1+\frac {\epsilon }{2})
    \big (\frac {n-\ell }{n} \big ) ^2.$$
    By Corollary \ref{corollary: control of submartingale} we find $$\Pr
    \big [Y_{u,v} \geq (1+{\epsilon })\frac{t(n-\ell )^2}{n^2} \big ]
    \leq e^{-\epsilon ^2 t(n-\ell )^2/16n^2}=o(1/n^4).$$ Here we used that
    $t(n-\ell )^2 /n^2 \gg \log n$.
    By applying the union bound we obtain (\emph{c}).
    \vspace{1mm}

    \noindent \textbf{Stage 2:} Completing the copies of
    $P_1,\ldots , P_t$ to copies of $C_1,\ldots ,C_t$.
    \vspace{1mm}

    Let us suppose that in Stage 1 we found $Q_1,\ldots ,Q_t$ in $D_1 \sim
    {\cal D}(n,p_1)$, and that property (\emph{c}) above holds.
    In this stage we will prove that w.h.p. it is possible to use edges of
    $D_2 \sim {\cal D}(n,p_2)$ to complete each oriented path $Q_i$ to a copy of $C_i$ which is edge-disjoint to other $C_j$-s.

    To see this, for each $i\in [t]$ let $W_i = V(D) \setminus \{x_{i,2}, \ldots ,
    x_{i,\ell -1}\}$ (recall that $Q_i = x_{i,1}\ldots x_{i,\ell }$). Let
    $G_i$ denote the digraph on vertex set $W_i$ consisting of all directed edges
    which do not lie in the paths $P_1,\ldots P_t$. Clearly we have
    $\delta ^{0}(G_i) \geq |W_i| - 2t \geq |W_i| - 2\Delta = (1-o(1))|W_i|$. Also, by property (\emph{c}) from Stage 1
    for each $\overrightarrow {uv} \in E(G_i)$ we have $Y_{u,v} \leq
    (1 + \epsilon ) {t(n-\ell )^2}/{n^2}$.

    Now select $D_2 \sim {\cal D}(n,p_2)$, where $p_2 = (1+o(1))\epsilon p$.
    Given $D_2$, we obtain a random subdigraph $F_i$ of $G_i$ by assigning
    $\overrightarrow{uv} \in E(D_2)$ with probability $1/Y_{u,v}$ to some
    $F_i$ with $\{u,v\} \subset V(W_i)$. By (\emph{c}), each edge of
    $G_i$ appears independently in $F_i$ with probability
        \begin{equation*}
            \frac{p_2}{Y_{u,v} }
                \geq
            \frac{\epsilon pn^2}{{(1+\epsilon )t(n-\ell )^2} }
                \geq
            \frac {\epsilon n}{2(n-\ell )^2} := p_{in}.
        \end{equation*}
    Therefore the distribution of $F_i$ stochastically dominates that of
    $H_i \sim {\cal D} (G_i,p_{in})$.

    Now to complete the proof, let $P_i^c$ denote the complimentary path to $P_i$ in $C_i$.
    Using $n-\ell =  n/\alpha \log n$, we find
        \begin{equation*}
            p _{in} \geq  \frac {\epsilon \alpha  \log  n }{2(n-\ell )}
            \gg \frac {\log |W_i| }{|W_i|}.
        \end{equation*}
    Therefore we can apply Lemma \ref{finishing a long path} to obtain that with
    probability $1 - o(1/n)$ for all $i\in [t]$, the digraph $H_i$
    (and therefore also $F_i$) contains a copy of $P_i^c$ from $x_{i,1}$
    to $x_{i,\ell }$ in $W_i$, denoted $Q_i^c$. But combining
    $Q_i$ with $Q_i^c$ for each $i\in [t]$ we obtain a copy of $C_i$.
    Therefore w.h.p., for all $i\in [t]$, the digraph $Q_i \cup F_i$
    contains a copy of $C_i$.
	\vspace{1mm}
	
    Stage 1 and 2 together prove that if $D_1 \sim {\cal D}(n,p_1)$ and
    $D_2 \sim {\cal D}(n,p_2)$ then with probability at least $1 - o(n^{-1})$
    the digraph $D_1 \cup D_2$ contains edge disjoint
    copies of $C_1,\ldots ,C_t$. As $D_1 \cup D_2$ can be coupled as a subgraph
    of $D \sim {\cal D}(n,p)$. This proves that the theorem holds with
    probability $1 - o(n^{-1})$ for $\log ^3 n \ll p \ll n^{-2/3}$, and
    therefore by the reduction mentioned at the beginning,
    w.h.p. for all $p \gg \log ^3n/n$.

\section{Counting}

    In this section we prove Theorem \ref{thm: Counting}.

    \begin{proof}[Proof of Theorem \ref{thm: Counting}]
    To prove the theorem, suppose that $p = \alpha ^2(\log \log n) \log n/n$
    for some function $\alpha = \alpha (n)$ which tends to infinity
    with $n$. Let us also set $\ell = n - n/\alpha (\log \log n)$.

    Let $C$ be an $n$-vertex $\sigma $-cycle for some $\sigma \in
    \{+,-\}^n$. Let $\rho \in \{+,-\}^{\ell }$ denote the vector
    given by $\rho (i) = \sigma (i)$ for all $i\in [\ell ]$ and let
    $P$ denote the $\rho $-subpath of $C$. Let us set
    $p_1 = (1-\epsilon )p$ and $p_2 = \epsilon p$, for fixed small constant
    $\epsilon >0$. We prove that $D \sim {\cal D}(n,p)$
    contains many copies of $C$ in two stages. In the first stage we
    show that w.h.p.
    $D_1 \sim {\cal D}(n,p_{1})$ contains $(1-o_{\epsilon}(1))^nn!p^n$ copies
    of $P$. In the second stage, we expose a further random digraph
    $D_2 \sim {\cal D}(n,p_2)$ and show that w.h.p. `most' of the
    copies $Q$ of $P$ in $D_1$ extend to a copy of $C$ in $D_2 \cup
    Q$.
    \vspace{2mm}

    \noindent \textbf{Stage 1:} $D_1 \sim {\cal D}(n,p_1)$ contains
    at least $(1-3\epsilon )^nn! p^n$ copies of $P$ w.h.p.
    \vspace{2mm}

    To begin, consider the following way to select
     a random copy of $P$, denoted $Q = x_1\cdots x_{\ell }$, in
    some \emph{fixed} digraph $D$ on $n$ vertices.
        \begin{enumerate}
            \item In the first round, select a vertex $x_1 \in V(D)$
            uniformly at random and set $Q_1 := x_1$.

            \item Suppose now that we are in round $i$, for some
            $1\leq i \leq \ell -1$ and so far we have found
            $Q_i = x_1\cdots x_i$ and aim to extend it to
            $Q_{i+1}$, by selecting $x_{i+1}$. Let $R_i$
            denote the $\sigma (i)$-neighbourhood of $x_i$
            in $V(D) \setminus V(Q_i)$, i.e. $R_i = N^{\sigma (i)}
            (x_i) \cap \big ( V(D) \setminus V(Q_i) \big )$.

            \item Select a vertex uniformly at random from
            $R_{i}$ and set it equal to $x_{i+1}$ and
            $Q_{i+1} := Q_ix_{i+1}$. If no such vertex exists declare
            a failure and terminate the algorithm. If
            $i < \ell - 1$, return to 1. for round $i+1$.

            \item If $i = \ell -1$, output $Q:=Q_{\ell }$.
        \end{enumerate}
    Running this randomized algorithm results in a distribution
    on the set of all $\rho $-paths $Q$ in $D$. We will write
    ${\cal F}(D)$ for this distribution.

    We will now analyse the above algorithm while running on
    ${\cal D}(n,p_1)$. Select $D_1 \sim {\cal D}(n,p_1)$
    and $Q \sim {\cal F}(D_1)$. For each $i\in [\ell ]$, we will be
    interested in the following event:
        \begin{align*}
        E_i &= ``
             |R_j| \geq (1-\epsilon )(n-i)p_1 \mbox{ for all }j<i".
    \end{align*}
    Note that if the algorithm ended in failure, $E_{\ell}^c$ must occur. We claim that
        \begin{equation}
        \label{equation: prob succeeding through low prob events}
            \Pr
            _{\substack{D_1\sim {\cal D}(n,p_{1})\\ Q \sim {\cal F}(D_1)}}
            \big [ E_{\ell } \big ] = 1 - o(1).
        \end{equation}

    To see this, we analyse the algorithm by generating $D_1$
    in an `online fashion', exposing edges as we go. Suppose now that
    we are in round $i$ of the algorithm and have so far found
    $Q_i = x_1\cdots x_i$. Expose all edges of $D_1$ in
    direction $\sigma (i)$
    between $x_i$ and $V(D_1)\setminus V(Q_i)$. Note that under this
    process, each edge is exposed at most once, and so can be coupled
    as a subgraph of $D_1 \sim {\cal D}(n,p_1)$. Clearly with this
    process, $|R_i| \sim \Bin (n-i,p_{1})$. Therefore, by Chernoff's
    inequality, we have
        \begin{equation*}
            \Pr
            _{\substack{D_1 \sim {\cal D}(n,p_{1})\\ Q \sim {\cal F}(D_1)}}
            \big [ |R_i| < (1-\epsilon )(n-i)p_1 \mbox{ } \big |\mbox{ } Q_i \big ] \leq e^{- 2\epsilon ^2(n-i)p_1} \leq
            e^{- 2\epsilon ^2(n-\ell )p_1} = o(n^{-1}).
        \end{equation*}
    Here we have used that $$\epsilon ^2(n-\ell )p_1 \geq \epsilon ^2
    (n- \ell )p/2 \geq \alpha \log n /2 \gg \log n.$$
    However, this gives that
    $$\Pr _{\substack{D_1 \sim {\cal D}(n,p_{1})\\ Q \sim {\cal F}(D_1)}}
            \big [ E_{i+1} |\mbox{ } E_i \big ] \geq  1 - o(n^{-1}).$$
    In turn this gives \eqref{equation: prob succeeding through low
    prob events}, since $\ell \leq n$ and
            \begin{equation*}
            \Pr
            _{\substack{D_1 \sim {\cal D}(n,p_{1})\\ Q \sim {\cal F}(D_1)}}
            \big [ E_{\ell } \big ]
                \geq
            \prod _{i\in [\ell -1]}
            \Pr
            _{\substack{D_1 \sim {\cal D}(n,p_{1})\\ Q \sim {\cal F}(D_1)}}
            \big [ E_{i+1} |\mbox{ } E_i \big ]
                \geq
            \big (1 - o (n^{-1} )\big )^{\ell -1} = 1 - o(1).
        \end{equation*}

    Now note that \eqref{equation: prob succeeding through low prob events} shows that if we select $D_1 \sim {\cal D}(n,p_{1})$ then w.h.p.
        \begin{equation*}
            \Pr _{Q \sim {\cal F}(D_1)} \big [ E_{\ell } \big ]
                =
            1- o(1).
        \end{equation*}
    However, for each $\sigma $-path
    ${\widetilde Q} = x_1\cdots x_{\ell }$ in $D_1$ which
    satisfies $E_{\ell }$ we have
        \begin{equation*}
            \Pr _{Q \sim {\cal F}(D_1)} (Q = {\widetilde Q})
                =
            \prod _{i\in [\ell -1]} \frac {1}{|R_i|} \leq
            \prod _{i\in [\ell -1]} \frac {1}{(1-\epsilon )(n-i)p_{1}}.
        \end{equation*}
    Therefore, letting ${\cal Q}(D_1)$ denote the collection of all
    $\sigma $-paths in $D_1$, which satisfy $E_{\ell }$, from
    \eqref{equation: prob succeeding through low prob events} we have
        \begin{equation*}
            1 - o(1)
                \leq
            \Pr _{Q \sim {\cal F}(D_1)}
            \big [ E_{\ell } \big ]
                =
            \sum _{{\widetilde Q} \in {\cal Q}}
            \Pr _{Q \sim {\cal F}(D_1)} \big [ Q = {\widetilde Q} \big ]
                \leq
            \frac {|{\cal Q}(D_1)|}{(1-\epsilon )^{\ell } (n)_{\ell }
             p_{1}^{\ell }}.
        \end{equation*}
    Rearranging, this gives $$|{\cal Q}(D_1)| \geq
    (1-\epsilon )^n (n)_{\ell }p_{1}^{\ell } \geq
    (1-2 \epsilon )^n (n)_{\ell } p^{\ell } \geq (1-3 \epsilon )^nn! p^n.$$
    Here we used that $$(n - \ell )! p^{n-\ell } \leq ((n-\ell )p)^{ n- \ell }
    = (\alpha \log n)^{n/\alpha \log \log n}=(1 + o(1))^n.$$
    \vspace{2mm}

    \noindent \textbf{Stage 2:} Completing `most' copies of $P$
    in $D_1$ to a copy of $C$.
    \vspace{2mm}

    Let ${\cal P}$ denote the collection of all copies of $P$  in
    $D_1$. Let $P^c$ denote compliment path of $P$ in $C$ (see notation).
    Note from the bound in Stage 1, w.h.p. we have
    $|{\cal P}| \geq (1- 3\epsilon )^n n!p^n$. Let  us fix
    $Q \in {\cal P}$, which starts at $x_1$ and ends at $x_{\ell }$.
    Select $D_2 \sim {\cal D}(n,p_2)$.
    We will show that
        \begin{equation}
            \label{equation: extending P-paths}
            \Pr \big [Q \mbox{ is contained in a copy of } C \mbox { in }
            Q\cup D_2 \big ] = 1-o(1).
        \end{equation}
    To see this, set $W_{Q} := \big (V(D_1) \setminus V(Q) \big ) \cup
    \{x_1 , x_{\ell } \}$, so that $|W_{Q}| = n - \ell + 2$. But it is easy to see
    that $D_2[W_Q] \sim {\cal D}(n-\ell +2, p_2)$ (perhaps some edges also appear
    $D_1$, but this only helps us). Using that
    $$p_2 = \epsilon p = \frac{\epsilon \alpha ^2(\log \log n) \log n}{n}
    \gg \frac{\log n }{n-\ell },$$ by Lemma \ref{finishing a long path} we find that
    $D_2[W_Q]$ w.h.p. contains a $P^c$ path $Q_2$ from $x_1$ to $x_{\ell }$.
    Combined with $Q$, this gives a copy of $C$ in $Q \cup D_2$. This gives
    \eqref{equation: extending P-paths}.

    We now complete the proof of the theorem. Let ${\cal P}_{bad}$ denote the set
    of $Q \in {\cal P}$ which are \emph{not} contained in a copy of $C$ in
    $Q \cup D_2$. From \eqref{equation: extending P-paths} we have
        \begin{equation*}
                {\mathbb E}(|{\cal P}_{bad}|) = o ( |{\cal P}| ).
        \end{equation*}
     By Markov's inequality this gives that w.h.p. $|{\cal P}_{bad }| =
     o(|{\cal P}|)$. Therefore, w.h.p. there are $|{\cal P} \setminus {\cal P}_{bad}|
     = (1-o(1)) |{\cal P}| \geq (1-4\epsilon )^n n!p^n$ paths $Q$ which extend to
     a copy of $C$ in $Q \cup D_2$. As each copy of $C$ can be obtained from at most
     $2n$ such paths, this gives $|{\cal P}\setminus {\cal P}_{bad}|/2n \geq
     (1 - 5\epsilon )^n n!p^n$ copies of $C$ in   $D_1 \cup D_2$.
    \end{proof}


\begin{thebibliography}{10}

\bibitem{bollobas1984evolution}
B.~Bollob{\'a}s.
\newblock The evolution of random graphs.
\newblock {\em Transactions of the American Mathematical Society},
  286(1), 257--274, 1984.

\bibitem{Chung-Lu} F. Chung and L. Lu. Concentration inequalities and
    martingale inequalities -- a survey, \textit{Internet Math.}
    \textbf{3}(1) (2006), 79-127.

\bibitem{Debiasio}
 L. DeBiasio, D. K\"uhn, T. Molla, D. Osthus and A. Taylor.
 \newblock Arbitrary orientations of Hamilton cycles in digraphs,
 \newblock in {\em SIAM Journal Discrete Mathematics} 29 (2015), 1553--1584.

\bibitem{Debiasio1}
L. DeBiasio and T. Molla.
\newblock Semi-degree threshold for anti-directed Hamilton cycles,
\newblock {\em arXiv preprint}, arXiv:1308.0269.

\bibitem{Dirac}
G. A. Dirac.
\newblock Some theorems on abstract graphs,
\newblock in {\em Proceedings of
the London Mathematical Society}, 2 (1952), 69--81.

\bibitem{Closing gaps}
A. Ferber.
\newblock Closing gaps in problems related to Hamilton cycles in random graphs
  and hypergraphs,
\newblock {\em The Electronic Journal of Combinatorics}, 22(1):P1--61, 2015.


\bibitem{FerKroLong}
A. Ferber, G. Kronenberg and E. Long.
\newblock Packing, counting and covering Hamilton cycle in random directed graphs,
\newblock \emph{Israel Journal of Mathematics}, to appear.

\bibitem{ferberrobust}
A.~Ferber, R.~Nenadov, U.~Peter, A.~Noever, and N.~{\v{S}}koric.
\newblock Robust {H}amiltonicity of random directed graphs.
\newblock {\em SODA '14}.


\bibitem{FV}
A. Ferber and V. Vu.
\newblock Packing perfect matchings in random hypergraphs,
\newblock preprint.

\bibitem{frieze1988algorithm}
A.~M. Frieze.
\newblock An algorithm for finding {H}amilton cycles in random directed graphs.
\newblock {\em Journal of Algorithms}, 9(2):181--204, 1988.

%\bibitem{karp1972reducibility}
%R.~M. Karp.
%\newblock Reducibility among combinatorial problems.
%\newblock Springer, 1972.
\bibitem{Ghouila}
 A. Ghouila-Houri.
 \newblock Une condition suffisante d’existence d’un circuit Hamiltonien,
 \newblock in {\em C.R. Acad. Sci. Paris} 25
(1960), 495–-497.

\bibitem{glebov2013number}
R.~Glebov and M.~Krivelevich.
\newblock On the number of {H}amilton cycles in sparse random graphs.
\newblock {\em SIAM Journal on Discrete Mathematics}, 27(1):27--42, 2013.

\bibitem{Hagg}
R. H\"aggkvist and A. Thomason.
\newblock Oriented Hamilton cycles in digraphs,
\newblock in {\em Journal of Graph Theory} 19.4 (1995), 471--479.

\bibitem{knox2013edge}
F.~Knox, D.~K{\"u}hn, and D.~Osthus.
\newblock Edge-disjoint {H}amilton cycles in random graphs.
\newblock {\em Random Structures \& Algorithms}, 2013.

\bibitem{krivelevich2012optimal}
M.~Krivelevich and W.~Samotij.
\newblock Optimal packings of {H}amilton cycles in sparse random graphs.
\newblock {\em SIAM Journal on Discrete Mathematics}, 26(3):964--982, 2012.

%\bibitem{kuhn2013hamilton}
%D.~K{\"u}hn and D.~Osthus.
%\newblock {H}amilton decompositions of regular expanders: a proof of kelly’s
%  conjecture for large tournaments.
%\newblock {\em Advances in Mathematics}, 237:62--146, 2013.

\bibitem{kuhn2014hamilton}
D.~K{\"u}hn and D.~Osthus.
\newblock {H}amilton decompositions of regular expanders: applications.
\newblock {\em Journal of Combinatorial Theory, Series B}, 104:1--27, 2014.


\bibitem{McDiarmid}
C.~McDiarmid.
\newblock Clutter percolation and random graphs,
\newblock in {\em Combinatorial Optimization II}, 17--25. Springer, 1980.

\bibitem{posa}
L.~P{\'o}sa.
\newblock {H}amiltonian circuits in random graphs.
\newblock {\em Discrete Mathematics}, 14(4):359--364, 1976.

\bibitem{Thom}
A. Thomason.
\newblock Paths and cycles in tournaments,
\newblock in {\em Transactions of the American Mathematical Society}, 296(1) (1986), 167--180.


\end{thebibliography}
\end{document}